\documentclass[a4paper,fleqn]{amsart}
\usepackage{graphicx}
\usepackage{amssymb}
\usepackage{amsmath}
\usepackage{amsthm}
\usepackage{mathrsfs}
\usepackage{amsfonts}
\usepackage{amscd}
\usepackage{epsfig}
\usepackage{epstopdf}
\usepackage{color}
\usepackage{graphicx}
\usepackage{cite}
\newtheorem{thm}{Theorem}[section]

\newtheorem{lem}[thm]{Lemma}
\newtheorem{prop}[thm]{Proposition}
\theoremstyle{definition}
\newtheorem{defn}[thm]{Definition}
\theoremstyle{remark}

\numberwithin{equation}{section}
\newcommand{\abs}[1]{\left\vert#1\right\vert}

\newcommand{\R}{{\mathbb R}}

\newcommand{\Z}{{\mathbb Z}}
\newcommand{\C}{{\mathbb C}}
\newcommand{\F}{{\mathbb F}}
\title{Mod 2 cohomology ring of a kind of orbit configuration space}
\author{Hao Li}

\begin{document}
\maketitle
\begin{abstract}
In this paper we caculate mod 2 cohomology ring of $F_{\mathbb{Z}_2^m}(\mathbb{R}^m,n)$ , which is local representation of orbit congfiguration spaces over small covers. We construct a differntial graded algebra, and there is a ring isomorphism between its mod 2 cohomology ring and  $H^*(F_{\mathbb{Z}_2^m}(\mathbb{R}^m,n),\mathbb{Z}_2)$. This idea can also be applied to calculate mod 2 cohomology ring of complement space of real arrangements.
\end{abstract}
\section{Introduction}
Let X be a topological space. We will consider configuration space of n ordered distinct points in X:
$$F(X,n)=\{(x_1,\ldots,x_n)\in X^n| x_i\neq x_j,\quad \forall  \quad i\neq j \}$$
If X admits a group action $G\times X \rightarrow X$, we consider orbit configuration space:
$$F_G(X,n)=\{(x_1,\ldots,x_n)\in X^n| G(x_i)\cap G(x_j)=\emptyset,\quad \forall  \quad i\neq j \}$$

The notion of configuration space was introduced in physics in 1940's. In mathematics, configuration spaces were first introduced by Fadell and Neuwirth\cite{[FN]} in 1962.   

The classical configuration space is $F(\R^2, n)\cong F(\C,n)$, it is exactly the complement space of the union of finite hyperplanes in $\C^n$, its fundamental group eqauls to classical pure braid group. In 1969, Arnol'd\cite{[VA]} computed the cohomology ring of $F(\C,n)$, it is the form of Orlik-Solomon algebra in arrangement theory. $F(\R^k,n)$ is the complement of finite union of linear subspaces of codimension k in $\R^{nk}$, F.R.Cohen\cite{[CF]} calculated its integral cohomology ring as a free Lie algebras with each generator corresponding to a codimension-k subspace. In 2000, Feichtner and Ziegler\cite{[FZ00]} determined $H^*(F(S^k,n);\Z)$; in 2001\cite{[FZ02]}, they  computed $H^*(F_{\Z_2}(S^k,n);\Z)$ ,  where the group aciton is the antipodar map.

If X is a smooth complex projective variety. In 1994, Fulton-MacPherson\cite{[FM]} proved that the rational cohomology ring can be computed from the rational cohomology ring of X and the Chern class of X. Totaro\cite{[To]} improved their work by proving that Chern class is actually irrevalent.

M.A. Xicoténcatl\cite{MA} did a lot of work in his Ph.D. thesis  on orbit configuration space where $G$ acts freely on $M$, he computed the cohomology and loop space homology of  some free action spaces, such as complements of arrangements, spaces of polynomials and spaces of type $K(\pi, 1)$.

But for other topological spaces with non-free group action, it becomes much harder to compute their homotopy groups and cohomology rings, tools used in above examples can no longer be applied to the computation. 

There is a very typical kind of spaces with non-free group action. In 1991, Davis and Januszkiewicz\cite{[DJ]} introduced four classes of nicely behaving manifolds over simple convex polytopes---small covers, quasi-toric manifolds, (real) moment-angle manifolds which have become important objects in toric topology. We are interested to study the orbit configuration spaces $F_{G_d^m}(M,n)$ for a dm-dimensional $G_d^m$-manifold $M$ over a simple convex m-polytope P ---Where $M$ is a small cover and $G_d^m=\Z_2^m$ when $d=1$, and a quasi-toric manifold and $G_d^m=T^m$ when $d=2$. We expect to find the relation between the algebraic topology of $F_{G_d^m}(M,n)$ and combinatoric informations of polytope P. In 2008, Junda Chen\cite{[CLW]} gave an explicit formula for the Euler characteristic of $F_{G_d^m}(M,n)$ in terms of the h-vector of P and gave a description of homotopy type when $n=2$. But there is still some distance between this work and our expectation.

In this paper, we focus on $F_{\mathbb{Z}_2^m}(\mathbb{R}^m,n)$. Since it is the local representation of $F_{\mathbb{Z}_2^m}(M,n)$, the results in this paper will help to improve the study of $F_{\mathbb{Z}_2^m}(M,n)$. Besides, $\Z_2^m \curvearrowright \R^m$ is a typical non-free 
group action over Euclidean space, it can be an interesting example for the computation of orbit configuration spaces.

In the pointview of arrangements, $F_{\mathbb{Z}_2^m}(\mathbb{R}^m,n)$ can be regarded as complement space of a collection of subspaces in Euclidean space. The study of complex arrangement is complete. In 1982, Richard Randell \cite{Ran} gave a nice description of the fundamental group of complement of complexification of real arrangements. In 1995, De Concini and Procesi\cite{[DC]} constructed a rational model using only labeled lattice proved that the rational cohomology ring are determined by this lattice. In 1999, Sergey Yuzvinsky \cite{[Yu]} constructed rational model on atomic complex to simplify De Concini and Procesi's result.

There is little result for real arrangements. Goresky-MacPherson describe the integral homology group. In 2000, Mark de Longueville and Carsten A. Schultz \cite{[LS]} computed the integral cohomology ring of geometric ($\geq 2$) real arrangements. For general real arrangements,  there isn't a good tool to get its cohomology ring information. 

Our main result is an description of $H^*(F_{\mathbb{Z}_2^m}(\mathbb{R}^m,n);\Z_2)$. Let $\mathcal{A}$ be a real arrangement in Euclidean space, $M(\mathcal{A})$ denote its complement space. A differential graded algebra $\widetilde{D}$ is constructed based on the intersection poset of arrangement $\mathcal{A}$. And $H^*(\widetilde{D},\Z_2)$ is computable.
\begin{thm}
There is a ring isomorphism $H^*(F_{\mathbb{Z}_2^m}(\mathbb{R}^m,n);\Z_2)\cong H^*(\widetilde{D},\Z_2)$
\end{thm}

This method can be applied to calculation of mod 2 cohomology ring of any real arrangements. 

This paper is organised as follows. In section 2 and 3, we give a brief introduction on the notions of small covers and quasi-toric manifolds, subspace arrangements, Goresky-MacPherson isomorphism and intersection product in arrangements theory; in section 4, we construct a differential graded algebra to describe the cohomology ring of $H^*(F_{\mathbb{Z}_2^m}(\mathbb{R}^m,n);\Z_2)$; in section 5, we give a simple example.
\section{Preliminary}
\subsection{Small covers and quasi-toric manifolds}
By the definitions in \cite{[DJ]}, let $P^m$ be an $m$-dimensional simple convex polytope. 

Let $G_d^m$ be $\Z_2^m$, $\F_d=\R$ if d=1; and the torus $T^m$, $\F_d=\C$ if $d=2$.

The natural action of $G_d^m$ on $\F_d^m$ is called the \textit{standard representation} , and the orbit space is $\R^m_+$.

 A $dm$-dimensional $G_d^m$-manifold $M^{dm}$ over $P^m$, is a smooth closed $dm$-dimensional manifold $M^{dm}$ with a locally standard $G_d^m$-action such that the orbit space is $P^m$. 
 
 A $G_d^m$-manifold $M^{dm}$ is called \textit{small cover} if $d=1$ and \textit{quasi-toric manifold} if $d=2$. 
\subsection{Local representation}
Since it is difficult to describe the topology of $F_{G_d^m}(M^{dm},n)$, we first consider its local representation.

For $d=1$,the orbit configuration space is $F_{\Z_2^m}(M^m,n)$, its local representation is $F_{\Z_2^m}(\R^m,n)$.

For $d=2$, the orbit configuration space is $F_{T^m}(M^{2m},n)$, its local representation is $F_{T^m}(\R^{2m},n)$.

In this paper, we only consider the case when d=1.

We can observe that\\ $F_{\Z_2^m}(M^{m},n)=\{(x_1,x_2,...,x_n)\in (\R^m)^n|x_i \in \R^m, \Z_2^m(x_i)\cap \Z_2^m(x_j)=\emptyset , \forall i \neq j\}\\ \hspace*{1.9cm}
=(\R^m)^n\setminus \underset{1\leq i < j\leq n}{\bigcup} \{(x_1,x_2,...,x_n)\in (\R^m)^n|\Z_2^m(x_i)=\Z_2^m(x_j)\}$

Let $A_{ij}\triangleq \{(x_1,x_2,...,x_n)\in (\mathbb{R}^m)^n|x_i\in \mathbb{Z}_2^m(x_j)\}$, it is the union of $2^m$ subspaces in $(\mathbb{R}^m)^n$, each subspace is in the form $$A_{ij}^g\triangleq \{(x_1,x_2,...,x_n)\in (\mathbb{R}^m)^n|x_i=g(x_j),g \in \mathbb{Z}_2^m\}$$ with codimension $m$.

Thus $F_{\mathbb{Z}_2^m}(\mathbb{R}^m,n)$ can be regarded as the complement space of subspace arrangement $\mathcal{A}$ in $(\mathbb{R}^m)^{\times n}$. $\mathcal{A}$ consists of $C_n^2 \times 2^m$ subspaces with codimension $m$. $\mathcal{A}=\{A_{ij}^g|1\leq i<j\leq n, g \in \mathbb{Z}_2^m\}$.

\subsection{Arrangement theory}
In this part, we review some useful concepts about subspace arrangement in \cite{[LS]}.
\subsubsection{\textbf{Notations}}
Let $\mathcal{A}$ be a linear subspace arrangement in a finite-dimensional $\mathbb{R}$-vector space W, let $\mathit{u}\subseteq \mathit{v}\subseteq W$ be linear subspaces.
\begin{itemize}
\item$\pi^\mathit{u}$: the quotient map $W \rightarrow W/\mathit{u}$
\item$\pi^{\mathit{u},\mathit{v}}$: the quotient map $W/\mathit{u} \rightarrow W/\mathit{v}$
\item$\mathcal{A}_\mathit{u}\triangleq\{\mathit{z}\in \mathcal{A}|\mathit{u}\subseteq \mathit{z}\}$ the subarrangement in W
\item$\mathcal{A}/\mathit{u}\triangleq \{\pi^\mathit{u}(\mathit{z})|\mathit{z}\in \mathcal{A}_\mathit{u}\}$ the arrangement in $W/\mathit{u}$
\item$M(\mathcal{A})$ denote the complement space $W \setminus \bigcup \mathcal{A}$
\item $P$  intersection poset, denote the set of all intersections of subset of $\mathcal{A}$.\\
The intersection poset $P$ is partially ordered by the reverse inclusion;\\
It has maximal element $\top \triangleq \bigcap \mathcal{A}$ and minimal element $W \triangleq \bigcap \emptyset$; \\
The join operation $\vee$ in $P$ is given by intersection;\\ 
$P$ is furnished with a dimension function $d:P \rightarrow \mathbb{N}$; \\
For $\mathit{u},\mathit{v} \in P$, we denote by $[\mathit{u},\mathit{v}],(\mathit{u},\mathit{v}],[\mathit{u},\mathit{v})$ the respective intervals in $P$.
\end{itemize}

For any partially ordered set Q, denote by $\triangle(Q)$ the order complex of Q whose simplices are given by chains in Q.

\subsubsection{\textbf{generic points}}
To establish a map between intersection poset $P$ and subspace arrangement, we have to introduce the concept of generic points.

For $\mathit{u}\in P$, generic point $x^\mathit{u}$ means either a point in $\mathit{u}\setminus \underset{\mathit{z}\in (\mathit{u},T]}\bigcup \mathit{z}$ or a map
\\ \begin{tabular}{rcc}
$x^{\mathit{u}}: [W,\mathit{u}]$&$\rightarrow$&$W/\mathit{u}$\\
$\mathit{v}$&$\mapsto$&$x_\mathit{v}^\mathit{u}$
\end{tabular} 
with $x_\mathit{v}^\mathit{u}\in \pi^\mathit{u}(v)\setminus \underset{\mathit{v}^\prime \in (\mathit{v},\mathit{u}]}\bigcup \pi^\mathit{u}(\mathit{v}^\prime)$ 

 Let $\mathit{u}\in P$, with generic points $x^\mathit{u}$, define a affine map
$\phi^{x^\mathit{u}}: \triangle[W,\mathit{u}]\rightarrow W/\mathit{u}$ which is affine on simplices and satisfies\\ 
$\phi^{x^\mathit{u}}(\mathit{w})=x_\mathit{w}^\mathit{u}\qquad \mathit{w}\in [W, \mathit{u}]$ 

Note that we can identify the abstract simplicial complex with its geometric realization.

\section{Goresky-MacPherson isomorphism and products}

This chapter mainly states the results in \cite{[LS]}. \subsection{Goresky-MacPherson isomorphism}
Given an arrangement $\mathcal{A}$ in W. If $\epsilon>0$ is a real number. $B_\epsilon$ denotes the open $\epsilon$-ball in W. Then
\begin{equation}
H^k(W \setminus \bigcup \mathcal{A})\xrightarrow[\cong]{i^*}H^k(B_\epsilon \setminus\bigcup\mathcal{A})\xrightarrow[\cong]{\cap[W]}H_{d(W)-k}(W, \bigcup\mathcal{A}\cup\mathcal{C}B_\epsilon)
\end{equation}
here $\mathit{i}$ denotes inclusion, [W] is the orientation class of W, and $\mathcal{C}B_\epsilon$ is the complement of $B_\epsilon$. 

The first isomorphism is trivial induced by inclusion map, the second isomorphism derives from Alexander Duality.

Now if we want to describe the cohomology ring of $M(\mathcal{A})$, we will work mainly in $H_*(W,\bigcup\mathcal{A}\cup\mathcal{C}B_\epsilon)$ with intersection product $\bullet$ given by $(\alpha \cap [W])\bullet (\beta \cap[W])=(\alpha \cup \beta)\cap [W]$, $\quad\alpha,\beta \in H^*(B_\epsilon \setminus\bigcup\mathcal{A})$.

Recall the map $\phi^{x^\mathit{u}}:\triangle[W,\mathit{u}]\rightarrow W/\mathit{u}$ 

For a simplex $\sigma \in \triangle[W,\mathit{u}]$.    $\sigma=<v_0,\dots,v_k>,v_0<\dots<v_k$, $ \phi^{x^\mathit{u}}(\sigma) \subseteq \pi^\mathit{u}({v_0})$, $\phi^{x^\mathit{u}}(\sigma)\cap \pi^\mathit{u}(v_k)=\{x_{v_k}^\mathit{u}\}$. 

$ \phi^{x^\mathit{u}}(\triangle(W,\mathit{u}])\subseteq \bigcup\mathcal{A}/\mathit{u}$, $\phi^{x^\mathit{u}}(\triangle[W,\mathit{u}))\subseteq W/\mathit{u}\setminus B_\epsilon^{W/\mathit{u}}$ (for small enough $\epsilon$).

Then we have map of pairs 
$$\phi^{x^\mathit{u}}: (\triangle[W,\mathit{u}],\triangle(W,\mathit{u}]\cup\triangle[W,\mathit{u}))\rightarrow(W/\mathit{u},\bigcup\mathcal{A}/\mathit{u}\cup\mathcal{C}B_\epsilon^{W/\mathit{u}})$$

We simplify $\Delta\Delta[W,\mathit{u}]\triangleq(\triangle[W,\mathit{u}],\triangle(W,\mathit{u}]\cup\triangle[W,\mathit{u}))$

Because $ (\pi^\mathit{u})^{-1}(\cup\mathcal{A}/\mathit{u})=\cup\mathcal{A}_\mathit{u}\subseteq\mathcal{A}$ and $(\pi^\mathit{u})^{-1}(\mathcal{C}B_\epsilon^{W/\mathit{u}})\subset \mathcal{C}B_\epsilon$

Thus we can consider the following maps:
\begin{equation}
H_k(\Delta\Delta[W,\mathit{u}])\xrightarrow{\phi_*^{x^\mathit{u}}}H_k(W/\mathit{u},\bigcup\mathcal{A}/\mathit{u}\cup\mathcal{C}B_\epsilon^{W/\mathit{u}})\xrightarrow{\pi_!^\mathit{u}}H_{d(\mathit{u})+k}(W,\bigcup\mathcal{A}\cup\mathcal{C}B_\epsilon)
\end{equation}
$\pi_!^\mathit{u}$ is given by $\alpha \cap [W/\mathit{u}]\mapsto (\pi^\mathit{u})^*\alpha\cap[W] \qquad \alpha \in H^*(W/\mathit{u})$

Then we can introduce well known Goresky-MacPherson isomorphism.
\begin{thm}[\textbf{Goresky-MacPherson isomorphism}]
Let $\mathcal{A}$ be an arrangement in W and $x^\mathit{u}$ be a choice of gneric points. Then the map 

$\sum_{\mathit{u}\in[W,\top]}\pi_!^\mathit{u}\circ \phi_*^{x^\mathit{u}}:\bigoplus_{\mathit{u}\in[W,T]}H_*(\Delta\Delta[W,\mathit{u}])\longrightarrow H_*(W,\bigcup\mathcal{A}\cup\mathcal{C}B_\epsilon)$
\\is an isomorphism as group.
\end{thm}
This proposition is originally proved in \cite{[MR]} by means of stratified Morse theory. There is another elementary proof in \cite{[LS]}.
\subsubsection{\textbf{Products}}
Let $P,Q$ be two intersection posets. $\triangle(P\times Q)=\triangle(P)\times \triangle(Q)$

If $C_*$ denotes the ordered chain complex, there is the well known map

$\quad C_*(\triangle P)\bigotimes C_*(\triangle Q)\xrightarrow{\times}C_*(\triangle(P\times Q))$  given by
\\$\langle u_0,\dots,u_k \rangle \bigotimes\langle v_0,\dots,v_l \rangle \longmapsto \underset{\substack{0=i_0\leq\dots\leq i_{k+l}=k\\0=j_0\leq \dots \leq j_{k+l}=l\\ \forall r \quad  (i_{r-1},j_{r-1})\neq(i_r,j_r)}}{\sum}\sigma_{i,j}\langle(u_{i_0},v_{j_0}),\dots,(u_{i_{k+l}},v_{j_{k+l}})\rangle$
\\where the $\sigma_{i,j}$ are signs determined by $\sigma_{i,j}=1$ if k=0 or l=0 and by $\partial(a\times b)=\partial a \times b +(-1)^k a\times \partial b$.

Since $\triangle\triangle P\times \triangle\triangle Q=\triangle\triangle (P\times Q)$, this induces a product
 $$\times : H_*(\triangle\triangle P)\otimes H_*(\triangle\triangle Q)\longrightarrow H_*(\triangle\triangle (P\times Q))$$

Let $\mathcal{A}$ be an arrangement in W and $\mathit{u},\mathit{v}\in P$. 

To describe the products on $H_*(W,\bigcup\mathcal{A}\cup\mathcal{C}B_\epsilon)$, we have to know the products on $H_*(\triangle\triangle[W,\mathit{u}])$ and $H_*(\triangle\triangle[W,\mathit{v}]) \quad\forall \quad\mathit{u},\mathit{v}\in P$

When $\mathit{u}+\mathit{v}=W$, $(\pi^{\mathit{u}\cap\mathit{v},\mathit{v}},\pi^{\mathit{u}\cap\mathit{v},\mathit{u}}):W/(\mathit{u}\cap\mathit{v})\rightarrow W/\mathit{v}\times W\mathit{u}$ is an isomorphism. $\epsilon_{\mathit{u},\mathit{v}}$ be the degree of this linear isomorphism.

The join operator \begin{center}
$\vee :$ $\begin{tabular}{rclcc}
$[W,\mathit{v}]$&$\times$&$[W,\mathit{u}]$&$\longrightarrow$&$[W,\mathit{v}\cap \mathit{u}]$\\
$(z$&,&$w)$&$\longmapsto$&$z\cap w$
\end{tabular}$
\end{center} induces a simplicial map of pairs
$\vee:\Delta\Delta[W,\mathit{v}]\times \Delta\Delta[W,\mathit{u}]\rightarrow\Delta\Delta[W,\mathit{u}\cap\mathit{v}]$.

The product is given in the same way as that of intersection posets defined above. We will get the following proposition.
\begin{prop}
Let $\mathcal{A}$ be an arrangement in W and $\mathit{u},\mathit{v}$ intersections in $\mathcal{A}$, such that $\mathit{u}+\mathit{v}=W$. Given generic points $x^\mathit{u}$ and $x^\mathit{v}$ we have for $a \in H_k(\Delta\Delta[W,\mathit{u}])$, $b \in H_l(\Delta\Delta[W,\mathit{v}])$, and the generic points $x^{\mathit{u}\cap\mathit{v}}$ constructed above, that
\\$\pi_!^\mathit{u}(\phi_*^{x^\mathit{u}}(a))\bullet\pi_!^\mathit{v}(\phi_*^{x^\mathit{v}}(b))=\epsilon_{\mathit{u},\mathit{v}}(-1)^{l(d(W)-d(\mathit{u}))}\pi_!^{\mathit{u}\cap\mathit{v}}(\phi_*^{x^{\mathit{u}\cap\mathit{v}}}(\vee_*(a\times b)))$.
\end{prop}

When $\mathit{u}+\mathit{v}\neq W$, there exists a non-trivial linear functional $\Lambda
:W\rightarrow \mathbb{R}$ with the kernel containing $\mathit{u}+\mathit{v}$. This induces functionals $\Lambda_\mathit{u},\Lambda_\mathit{v}$ on $W/\mathit{u},W/\mathit{v}$ respectively. then we can choose generic points $x^\mathit{u}$ and $y^\mathit{v}$ such that

\begin{tabular}{ccc}
$\Lambda_\mathit{u}(x_\mathit{w}^\mathit{u})\geq 0$&$\forall \mathit{w}\in (W,\mathit{u}],$&$\Lambda_\mathit{u}(x_\mathit{W}^\mathit{u})>0$,\\
$\Lambda_\mathit{v}(y_\mathit{w}^\mathit{v})\leq 0 $&$ \forall \mathit{w}\in (W,\mathit{v}],$&$\Lambda_\mathit{v}(y_\mathit{W}^\mathit{v})<0.$
\end{tabular}
\begin{prop}
Let $\mathcal{A}$ be an arrangemnet in W and $\mathit{u},\mathit{v}$ intersections in $\mathcal{A}$ with $\mathit{u}+\mathit{v}\neq W$. Then for generic points $x^\mathit{u}$ and $y^\mathit{v}$ constructed above, the composition 
\\$H_*(\Delta\Delta[W,\mathit{u}])\bigotimes H_*(\Delta\Delta[W,\mathit{v}])
\\ \hspace*{1cm} \xrightarrow{(\pi_!^\mathit{u}\circ \phi_*^{x^\mathit{u}})\bigotimes(\pi_!^\mathit{v}\circ \phi_*^{y^\mathit{v}})} H_*(W,\bigcup\mathcal{A}\cup\mathcal{C}B_\epsilon)\bigotimes H_*(W,\bigcup\mathcal{A}\cup\mathcal{C}B_\epsilon)\\
\hspace*{8cm}\xrightarrow{\quad\bullet\quad}H_*(W,\bigcup\mathcal{A}\cup\mathcal{C}B_\epsilon)$
\\is the zero map.
\end{prop}
The proof of Prop 3.2 and Prop 3.3 sees \cite{[LS]}.

\section{the cohomology of local representation}
In Prop 3.1, Prop 3.2 and Prop 3.3, Goresky-MacPherson isomorphism and intersection product of homology groups depend on the choice of generic points $x^\mathit{u}$. If the arrangements is a $(\geq 2)$-arrangement (that is $\forall \mathit{u},\mathit{v}\in P$, if $\mathit{u}<\mathit{v}$, then $d(\mathit{u})-d(\mathit{v})\geq 2$), by Lemma 5.1 in \cite{[LS]}, Goresky-MacPherson isomorphism is independent of the choice of generic points, in that case one can obtain a complete combinatorial description of the intersection product, furthermore one can describe the integral cohomology ring structure of $W \setminus \bigcup \mathcal{A}$.

However, it is not easy for real arrangements to satisfy $(\geq 2)$ condition. There are many cases for real arrangements such that the intersection of two subspaces decreases by only one dimension, and this is the key reason  why real arrangements are much more difficult to deal with than complex arrangements. 

In our concerning case, it is a pity that $F_{\mathbb{Z}_2^m}(\mathbb{R}^m,n)$ doesn't satisfy the $(\geq 2)$ condition, so we can't get rid of the influence of the choice of generic points. But if we consider the mod 2 cohomology ring instead of the integral cohomology ring, we can overcome this barrier. Thus we can get the following lemma easily.
\begin{lem}
Under $\mathbb{Z}_2$-coefficient, Goresky-MacPherson isomorphism and intersection products depend only on the intersection poset $P$.
\end{lem}
\begin{proof}
 For arbitary two generic points $x^\mathit{u}, \tilde{x}^\mathit{u}$, let $a\in C_k(\Delta\Delta[W,\mathit{u}])$, $\phi^{x^\mathit{u}}(a)$ and $\phi^{\tilde{x}^\mathit{u}}(a)$ differ merely by an orientation. Thus $\phi^{x^\mathit{u}}_*=\phi^{\tilde{x}^\mathit{u}}_* ( mod \quad2)$. It means the Goresky-MacPherson isomorphism and intersection product of homology groups are independent of the choice of generic points over $\Z_2$ coefficient.
\end{proof}
Furthermore, the intersection products can  be depicted as follows:
\begin{thm}
For an arrangement $\mathcal{A}$ with intersection poset $P$, under $\mathbb{Z}_2$-coefficients, the intersection product is given by the combinatorial data as follows.

\begin{tabular}{ccc}
$H_k(\Delta\Delta[W,\mathit{u}])\bigotimes H_l(\Delta\Delta[W,\mathit{v}])$&$\longrightarrow$&$ H_{k+l}(\Delta\Delta[W,\mathit{u}\vee\mathit{v}])$
\\$a\bigotimes b $&$ \longmapsto  $&$\Big\{\begin{matrix}
\vee_*(a\times b) & \mathit{u}+\mathit{v}=W
\\0 & otherwise
\end{matrix}$ 
\end{tabular}
\end{thm}
\begin{proof}
Immediately get by Prop 3.1, Prop 3.2 and Prop 3.3 via below sequence of maps

$H_*(\Delta\Delta[W,\mathit{u}])\bigotimes H_*(\Delta\Delta[W,\mathit{v}])\longrightarrow H_*(W, \bigcup\mathcal{A}\cup\mathcal{C}B_\epsilon)\otimes H_*(W, \bigcup\mathcal{A}\cup\mathcal{C}B_\epsilon)$\\
\hspace*{5.4cm}$\overset{\bullet}{\longrightarrow}H_*(W, \bigcup\mathcal{A}\cup\mathcal{C}B_\epsilon)$\\
\hspace*{5.4cm}$\overset{\cong}{\longleftarrow}\underset{\mathit{w}\in [W,\top]}{\bigoplus}H_*(\Delta\Delta[W,\mathit{w}])$
\end{proof}

In [11], Yuzvinsky construct differential graded algebra based on atomic complex to describe the rational cohomology ring of complement sapce of complex subspace arrangement. In this paper, we adapt his construction with different definition of order of elements.
 
Since $H_k(\Delta\Delta[W,\mathit{u}])\triangleq H_k([W,\mathit{u}],(W,\mathit{u}]\cup[W,\mathit{u}))$

Let $\sigma\in C_k(\Delta\Delta[W,\mathit{u}]), \sigma\sim (W<\mathit{u}_1<\mathit{u})$, $\mathit{u}_1\in C_{k-2}(\triangle(W,\mathit{u}))$

and $(W<\mathit{u}_1<\mathit{u})\sim (W<\mathit{u}_2<\mathit{u}) \Longleftrightarrow \mathit{u}_1\sim \mathit{u}_2 $

$\therefore H_k(\Delta\Delta[W,\mathit{u}])\cong \widetilde{H}_{k-2}(\triangle(W,\mathit{u})) \qquad k\geq 2$

$H_1(\Delta\Delta[W,\mathit{u}])=\bigg \{ \begin{matrix}
\mathbb{Z}_2 & \text{if u is an atom}
\\0  &   otherwise
\end{matrix}$
~\\

Now we determine the homology of $\triangle(W,\mathit{u})$.

Let $\mathcal{A}=\{A_1,\dots,A_p\}$ be an subspace arrangement, which are called atoms. $P$ be the intersection poset.

$\sigma \subset \mathcal{A}$. $\sigma=\{A_{i_1},\dots,A_{i_k}\}$, $\vee(\sigma)=A_{i_1}\cap\dots\cap A_{i_k}$.
 
Construct atomic complex $A(P)=\{\sigma\subset \mathcal{A}|\vee(\sigma)<\top \}$ ($\top \triangleq \bigcap \mathcal{A} $ ).

By Lemma 2.1 in \cite{[Yu]} , $A(P)$ is homotopy equivalent to order complex $\triangle(W,\top)$ 

For $\mathit{u}\in P$, define  $\mathcal{A}_u=\{A_i\in \mathcal{A}|\mathit{u}\subset A_i\}$. Its corresponding intesection poset is denoted by $P_\mathit{u}$. Then $A(P_\mathit{u})\simeq \triangle(W,\mathit{u})$ 

\begin{defn}
The relative atomic (chain) complex $D=D(P)$ is the free abelian group on all subsets $\sigma=\{A_{i_1},\ldots,A_{i_p} \} \subset \mathcal{A}$, $dim(\sigma)=|\sigma|$.with its differential 
\begin{center}
$\partial:$
\begin{tabular}{rcl}
$ C_n$&$ \rightarrow $&$ C_{n-1}$
\\$\sigma$&$\mapsto $&$\sum (-1)^j \sigma \setminus{A_{i_j}}$
\end{tabular}. 
\end{center}
where the summation is taken over index j such that $\vee(\sigma\setminus{A_{i_j}})=\vee(\sigma)$. 
\end{defn}
In fact, reative atomic complex $D$ can be represented as the direct sum of complexes. Let $\sum(\mathit{u})$ denotes the simplicial complex whose simplices are all the subsets of $\mathcal{A}_\mathit{u}$, denote $\overline{A(\mathit{u})}=\sum(\mathit{u})/A(P_\mathit{u})=\{\sigma\subset{\mathcal{A}_\mathit{u}}| \vee(\sigma)=\mathit{u}\}$, let $D(\mathit{u})=C(\overline{A(\mathit{u})})$. Obviously, $\sum(\mathit{u})$ is acyclic. The following lemma is immediate by easy calculation of simplicial homology group.

\begin{lem}
\begin{enumerate}
\item  $\widetilde{H}_p(D(\mathit{u}))\simeq \widetilde{H}_{p-2}(A(P_\mathit{u}))$ \item $D=\underset{\mathit{u}\in P}{\bigoplus}D(\mathit{u})$
\end{enumerate}
\end{lem}

So $\underset{\mathit{u}\in P}{\bigoplus}H_k(\Delta\Delta[W,\mathit{u}])\cong \underset{\mathit{u}\in P}{\bigoplus}\widetilde{H}_{k-2}(\triangle(W,\mathit{u}))\cong \underset{\mathit{u}\in P}{\bigoplus}\widetilde{H}_{k-2}(A(P_\mathit{u}))$

$\cong \underset{\mathit{u}\in P}{\bigoplus} \widetilde{H}_kD(\mathit{u})\cong \widetilde{H}_k(D)\cong H_k(D) \qquad k\geq 2$

When k=1, $\underset{\mathit{u}\in P}{\bigoplus}H_1(\Delta\Delta[W,\mathit{u}])\cong H_1(D)$

Combine them together, $\underset{\mathit{u}\in P}{\bigoplus}H_k(\Delta\Delta[W,\mathit{u}])\cong H_k(D)$ for $k\geq 1$
~\\

Now turn to the product, look at $D(\mathit{u})$, it is generated by $\sigma\in \mathcal{A}_\mathit{u}, s.t. \vee(\sigma)=\mathit{u}$

Recall two maps (2.1) and (2.2).

$H_k(\Delta\Delta[W,\mathit{u}])\xrightarrow{\pi_!^\mathit{u}\circ \phi_*^{x^\mathit{u}}}H_{d(\mathit{u})+k}(W, \bigcup \mathcal{A}\cup\mathcal{C}B_\epsilon)\xleftarrow[\cong]{\cap[W]\circ \mathit{i}^*}H^{d(W)-d(\mathit{u})-k}(W\setminus \bigcup \mathcal{A})$
  
To describe the cup product, we are to define a new differential graded algebra $\widetilde{D}$ .
\begin{defn}
differential graded algebra $\widetilde{D}$ is the free abelian group on all subsets $\sigma=\{A_{i_1},\ldots,A_{i_p} \} \subset \mathcal{A}$ let $deg(\sigma)=d(W)-\abs{\sigma}-d(\vee(\sigma))$, with its differential 
\begin{center}
$\delta:$
\begin{tabular}{rcl}
$ C^n$&$ \rightarrow $&$ C^{n+1}$
\\$\sigma$&$\mapsto $&$\sum (-1)^j \sigma \setminus{A_{i_j}}$
\end{tabular}. 
\end{center}
where the summation is taken over index j such that $\vee(\sigma\setminus{A_{i_j}})=\vee(\sigma)$.
and  define the multiplication on algebra $\widetilde{D}$ as following: $\sigma \circ \tau=\bigg \{ \begin{matrix}
\sigma\cup\tau & if \vee(\sigma)+\vee(\tau)=W
\\ 0 & otherwise
\end{matrix}$  
\end{defn}
  
Thus according to Goresky-MacPherson isomorphism and Alexander Duality, we get our result
\begin{thm}
there is a ring isomorphism $H^*(M(\mathcal{A});\Z_2)\cong H^*(\widetilde{D};\Z_2)$
\end{thm}
\begin{proof}
We have proved  $\underset{\mathit{u}\in P}{\bigoplus}H_k(\Delta\Delta[W,\mathit{u}])\cong H_k(D)$ as group,
together with Goresky-MacPherson isomorphism:
$$\sum_{\mathit{u}\in[W,\top]}\pi_!^\mathit{u}\circ \phi_*^{x^\mathit{u}}:\bigoplus_{\mathit{u}\in[W,T]}H_*(\Delta\Delta[W,\mathit{u}])\longrightarrow H_*(W,\bigcup\mathcal{A}\cup\mathcal{C}B_\epsilon)$$
we get $H_*(W,\bigcup\mathcal{A}\cup\mathcal{C}B_\epsilon)\cong H_*(D) $ as groups.

Since we have 
$$H^k(W \setminus \bigcup \mathcal{A})\xrightarrow[\cong]{i^*}H^k(B_\epsilon \setminus\bigcup\mathcal{A})\xrightarrow[\cong]{\cap[W]}H_{d(W)-k}(W, \bigcup\mathcal{A}\cup\mathcal{C}B_\epsilon)$$
And $H^*(\widetilde{D})$ is dual to $H_*(D)$, so $H^*(W \setminus \bigcup \mathcal{A})\cong H^*(\widetilde{D})$ as group.

Now we turn to product. Intersection product in Theorem 4.2 corresponds to cup product. In $H^*(\widetilde{D})$, the product defined in $H^*(\widetilde{D})$ dual to intersection product in $H_*(W, \bigcup\mathcal{A}\cup\mathcal{C}B_\epsilon)$ , therefore agree with cup product. so there is a ring isomorphism $H^*(W \setminus \bigcup \mathcal{A})\cong H^*(\widetilde{D})$.
\end{proof}

The shortcoming of rational model method is that we can not read the generators and relations explictly from the differential graded algebra. In my opinion, this shortcoming comes from the complexity of real arrangements.

\section{Example}
We take $F_{\mathbb{Z}_2^2}(\mathbb{R}^2,2)$ as an example to verify the Theorem 4.6.

Because $F_{\mathbb{Z}_2^2}(\mathbb{R}^2,2)=\C^2\setminus \mathcal{A} \qquad \mathcal{A}=\{H_1,H_2,H_3,H_4\}$ where 

\hspace*{1cm} $H_1=\{((x_1,y_1),(x_2,y_2))\in \C \times \C | x_1=x_2, y_1=y_2\}$

\hspace*{1cm} $H_2=\{((x_1,y_1),(x_2,y_2))\in \C \times \C | x_1=x_2, y_1=-y_2\}$

\hspace*{1cm} $H_3=\{((x_1,y_1),(x_2,y_2))\in \C \times \C | x_1=-x_2, y_1=-y_2\}$

\hspace*{1cm} $H_4=\{((x_1,y_1),(x_2,y_2))\in \C \times \C | x_1=-x_2, y_1=y_2\}$

The differential graded algebra $\widetilde{D}$ is stated as following:

\hspace*{0.5cm}$deg(\emptyset)=4-0-4=0$

\hspace*{0.5cm}$\delta(\emptyset)=0$
~\\

\hspace*{0.5cm}$deg(\sigma)=d(W)-\abs{\sigma}-d(\vee(\sigma))$

\hspace*{0.5cm}$deg(1)=deg(2)=deg(3)=deg(4)=4-1-2=1$

\hspace*{0.5cm}$\delta(1)=\delta(2)=\delta(3)=\delta(4)=0$
~\\

\hspace*{0.5cm}$deg(12)=deg(14)=deg(23)=deg(34)=4-2-1=1$

\hspace*{0.5cm}$deg(13)=deg(24)=4-2-0=2$

\hspace*{0.5cm}$\delta(12)=\delta(23)=\delta(34)=\delta(14)=\delta(24)=\delta(13)=0$
~\\

\hspace*{0.5cm}$deg(123)=deg(124)=deg(234)=deg(134)=4-3-0=1$

\hspace*{0.5cm}$\delta(123)=\delta(134)=13, \hspace*{0.5cm}\delta(124)=\delta(234)=24$
~\\

\hspace*{0.5cm}$deg(1234)=4-4-0=0$

\hspace*{0.5cm}$\delta(1234)=123+124+134+234$

Through easy caculation, we choose (1),(2),(3),(4),(12),(14),(23),(24),(123)+(134) as the generators of $H^1(F_{\mathbb{Z}_2^2}(\mathbb{R}^2,2),\Z_2)$, $(\emptyset)$ as the generators of $H^2(F_{\mathbb{Z}_2^2}(\mathbb{R}^2,2),\Z_2)$, thus $H^k(F_{\mathbb{Z}_2^2}(\mathbb{R}^2,2),\Z_2)=\bigg \{  \begin{matrix}
0&k\geq 2\\
\Z_2^9& k=1\\
\Z_2& k=0
\end{matrix}$, where products vanish. 

On the other hand, since $F_{\mathbb{Z}_2^2}(\mathbb{R}^2,2)=\C^2\setminus \mathcal{A}$, it can  deforme retract to a torus with eight points removed, which has homotopy type of the wedge of nine circles. Thus, the caculation of $H^k(F_{\mathbb{Z}_2^2}(\mathbb{R}^2,2),\Z_2)$ agrees in two different ways ---geometric and combinatoric.

\renewcommand{\refname}{reference}

\end{document}